\documentclass[a4paper,11pt]{amsart}%

\usepackage[T1]{fontenc}
\usepackage[utf8]{inputenc}
\usepackage{ifpdf}

\usepackage{bm}
\usepackage{array}
\usepackage{breqn}
\usepackage{multicol}
\usepackage{amssymb}
\usepackage{amsmath}
\usepackage{synttree}
\usepackage{amsthm}
\usepackage{algorithm}
\usepackage[noend]{algorithmic}
\usepackage{todonotes}
\usepackage{amsfonts}
\usepackage{graphicx}
\usepackage{comment}
\usepackage{hyperref}%
\usepackage{enumitem}
\allowdisplaybreaks

\DeclareMathAlphabet{\mathsfit}{T1}{\sfdefault}{\mddefault}{\sldefault}
\SetMathAlphabet{\mathsfit}{bold}{T1}{\sfdefault}{\bfdefault}{\sldefault}

\theoremstyle{definition}
\theoremstyle{plain}
\newtheorem{defn}{Definition}[section]

\newtheorem{theorem}[defn]{Theorem}
\newtheorem{proposition}[defn]{Proposition}
\newtheorem{lemma}[defn]{Lemma}
\theoremstyle{remark}
\newtheorem{remark}[defn]{Remark}

\newcommand{\Qmat}{\mathsfit{Q}}
\newcommand{\Q}{{\mathbb Q}}

\newcommand{\R}{{\mathbb R}}

\newcommand{\MM}{{\mathcal{M}}}

\newcommand{\bx}{{\bm{x}}}
\newcommand{\by}{{\bm{y}}}
\newcommand{\bxy}{{\bm{xy}}}
\newcommand{\bt}{{\bm{t}}}

\author{Santiago Laplagne}
\email{slaplagn@dm.uba.ar}
\address{Instituto de C\'alculo, FCEN, Universidad de Buenos Aires - Ciudad Universitaria, Pabell\'on I - (C1428EGA) - Buenos Aires, Argentina}

\title[SOS decomposition of positive polynomials]{Sum of squares decomposition of positive polynomials with rational coefficients}

\begin{document}

\begin{abstract}
We present an example of a strictly positive polynomial with rational coefficients that can be decomposed as a sum of squares of polynomials over $\R$ but not over $\Q$. This answers an open question by C. Scheiderer posed as the second question in \cite[Section 5.1]{scheiderer}. We verify that the example we construct defines a nonsingular projective hypersurface, giving also a positive answer to the third question posed in that section.
\end{abstract}

\maketitle

\section{Introduction}
\label{section:introduction}
Given a polynomial $f \in \R[x_1, \dots, x_n]$, deciding if it is always positive or non-negative is a basic problem in real algebra, with important applications in optimization and other areas of mathematics. This is in general a difficult problem, and a usual relaxation is the problem of deciding if a polynomial can be decomposed as a sum of squares of polynomials, which provides a non-negativity certificate. This problem can be tackled using semidefinite programming, for which many numerical algorithms and fast implementations exist. B. Sturmfels asked the question if every polynomial with rational coefficients that is a sum of squares of polynomials allows always a decomposition with rational coefficients. This question is relevant in the problem of computing exact non-negativity certificates.

In \cite{scheiderer}, C. Scheiderer provides the first examples of polynomials with rational coefficients that can be decomposed as a sum of squares with real coefficients ($\R$-sos) but cannot be decomposed as a sum of squares of polynomials with rational coefficients ($\Q$-sos). The polynomials obtained in his construction always have common real roots, and the existence of these common roots is central to prove that there is no rational decomposition. A different class of examples have been constructed in \cite{laplagne} and \cite{CLS}. The strategy used there also produces polynomials with real roots. A natural question, posed by C. Scheiderer as part of \cite[Open Question 5.1]{scheiderer}, is whether there exist a polynomial over $\Q$ that can be decomposed as a sum of squares of polynomials over $\R$ but not over $\Q$ that is strictly positive, that is, it has no nontrivial real zeros. This problem is also discussed in \cite[Open Question 1.10]{vinzant}, where the authors comment on the difficulties related to this problem. In Section \ref{section:main} we present an example that gives a positive answer to this question.

A sharper question posed also by C. Scheiderer is whether there exist such polynomials defining a nonsingular hypersurface. In Section \ref{section:hyper}, we show that our example gives also a positive answer to this question.


\section{The cone of sums of squares}
\label{section:sos}

We set some notation and recall basic results that will be used in our construction (see \cite[Chapter 3 and 4]{BPT} for details and proofs).
Let $H_{n,d}$ be the vector space of homogeneous polynomials in $n$ variables of degree $d$ with coefficients in $\R$. Let $\Sigma_{n,2d} \subset H_{n,2d}$ be the cone of polynomials that can be written as a sum of squares of polynomials in $H_{n,d}$ and let $\partial \Sigma_{n,2d}$ be its boundary.

The dual cone $\Sigma_{n,2d}^{*}$ of $\Sigma_{n,2d}$ is the set of all linear functionals in the dual space $H_{n,2d}^*$ that are nonnegative on $\Sigma_{n,2d}$.
A polynomial $f \in \Sigma_{n,2d}$ is in $\partial \Sigma_{n,2d}$ if and only if there exists $\ell \in \Sigma_{n,2d}^{*}$ such that $\ell(f) = 0$.
For an element $\ell \in H_{n, 2d}^{*}$, we denote $B_\ell$ the associated bilinear form $B_\ell(p, q) = \ell(pq)$ for $p, q \in H_{n, d}$. $B_\ell$ is positive semidefinite if and only if $\ell \in \Sigma_{n, 2d}^{*}$. We can also associate to $\ell$ a quadratic form $Q_{\ell}(p) = \ell(p^2)$. If $f = p_1^2 + \dots + p_s^2$ is a sum of squares and $\ell \in \Sigma_{n,2d}^{*}$, then  $\ell(f) = 0$ implies $Q_{\ell}(p_i) = 0$ for all $1 \le i \le s$. For computations, it is convenient to fix a monomial basis $\MM$ of $H_{n,d}$ and represent $Q_{\ell}$ in the coordinates of $\MM$ as a matrix $\Qmat_{\ell} \in \R^{N \times N}$, where $N = \dim H_{n,d}$. Since $B_\ell$ is positive semidefinite, $Q_\ell(p) = 0$ implies $(p)_{\MM} \in \ker(\Qmat_{\ell})$. 
Moreover $p \in \ker(Q_{\ell})$ implies $B_{\ell}(p, q) = 0$ for all $q \in H_{n,d}$ (see \cite[Lemma 2.6]{blekherman}).

Let $m$ be the vector of all monomials in $\MM$. A polynomial $f \in H_{n,2d}$ is a sum of squares if and only if there is a positive semidefinite matrix $A \in \R^{N \times N}$ such that $f = m^T A m$. We call any such matrix a Gram matrix of $f$. If $f$ is in the boundary of the SOS cone then any Gram matrix of $f$ will be singular, and conversely, if $f$ is in the interior of the SOS cone then there exists an invertible Gram matrix for $f$.

By the following result, we know that any  strictly positive polynomial over $\Q$ which is $\R$-sos but not $\Q$-sos must lie in the boundary of the SOS cone.
\begin{theorem}[Theorem 1.2, \cite{hillar}] Let $f \in \Q[x_1, \dots, x_n]$. If there is an invertible Gram matrix for
$f$ with real entries, then there is a Gram matrix for $f$ with rational entries.
\end{theorem}

Constructions for strictly positive polynomials in the boundary of the SOS cone have been found by several authors. In \cite{blekherman}, G. Blekherman gives a general construction for forms of degree 6 in 3 variables and degree 4 in 4 variables. In \cite{BSV}, the authors extend the latter construction to more general varieties and in \cite{valdettaro}, the authors use those ideas to build concrete examples for polynomials of any degree and number of variables.

\section{Main construction}
\label{section:main}

In our construction we glue together two different polynomials. The first polynomial is a polynomial with rational coefficients that can be decomposed as a sum of squares of polynomial over $\R$ but not over $\Q$, which has real zeroes (that is, it is not strictly positive). The second polynomial is a strictly positive polynomial in the boundary of the SOS cone.

\subsection{A polynomial with no rational decomposition}
\label{section:block1}

The example we use as one of the blocks was obtained in \cite{CLS} by J. Capco, S. Laplagne and C. Scheiderer, simplifying the construction in \cite{laplagne}.
We recall the example here. It is a sum of squares of 3 polynomials in 4 variables with coefficients in $\Q(\sqrt[3]{2})$. Noting $\alpha$ the real root of $X^3 - 2$, the 3 polynomials are
\begin{align*}
p_1 &= (-4\alpha^2+4\alpha-2)x_0^2+x_1^2-2x_1x_2+2x_2x_3-2x_3^2, \\
p_2 &= (-4\alpha^2-4\alpha+6)x_0^2+x_1^2+2x_1x_2+2x_2x_3+2x_3^2, \\
p_3 &= 4\alpha x_0x_1+4x_0x_2+4\alpha^2 x_0x_3,
\end{align*}
%
%
%
%
%
%
and the sum of squares is
\begin{dmath*}
f = p_1^2 + p_2^2 + p_3^2 =40x_0^4+8x_0^2x_1^2+32x_0^2x_1x_2+64x_0^2x_1x_3+16x_0^2x_2^2+16x_0^2x_2x_3+32x_0^2x_3^2+2x_1^4+8x_1^2x_2^2+8x_1^2x_2x_3+16x_1x_2x_3^2+8x_2^2x_3^2+8x_3^4,
\end{dmath*}
a polynomial of degree 4 with rational coefficients.

This polynomial is not strictly positive, the real roots are the common real roots of $p_1, p_2, p_3$, with approximate affine representatives
$$r_1 = (1, -1.817, 0.452, 1.158) \text{ and } r_2 = (0, 0, 1, 0)$$
(the coordinates of $r_1$ are in an algebraic extension of $\Q$ of degree 12). For our construction, it is only important that there are no real roots with $x_2 = 0$, which can be verified by hand.

We will now prove that the polynomial $f$ does not allow a rational decomposition. The strategy used here is different from the original proof by the authors of \cite{CLS} and we will use this new proof to extend the example to our main problem.

\begin{proposition}\label{prop:block1}
Let $f \in \Q[x_0, x_1, x_2, x_3]$ be as above, then $f$ cannot be decomposed as a sum of squares with rational coefficients.
\end{proposition}

See \cite[Worksheet A]{nonRatWorksheets} for the computations in Maple \cite{maple}. To prove the claim, we study first which polynomials can appear in a sum of squares decomposition of $f$. We seek to compute a linear form $\ell \in \Sigma_{4,4}^*$ such that $\ell(f) = 0$. Since $\ell$ is required to be non-negative on squares, the associated quadratic form
$$Q_\ell : \R[x_0, x_1, x_2, x_3]_2 \rightarrow \R$$
must be positive semidefinite and, es explained in Section \ref{section:sos}, any polynomial in a sum of squares decomposition of $f$ must be in the kernel of $Q_\ell$.

If we can construct a form as above with no non-trivial rational polynomials in the kernel, then we can assure that $f$ has no rational decomposition. To achieve this, we will construct a form with kernel as small as possible.

We follow a similar strategy as in \cite[Propositions 5.2 and 5.8]{valdettaro}. The form $\ell$ must satisfy $\ell(p_i q) = 0$ for any $q \in H_{4,2}$. It is defined over $H_{4,4}$, which is a 35-dimensional space. The dimension of $H_{4,2}$ is 10, so for each $p_i$ we get 10 restrictions for $\ell$ . This would be a total number of $30$ restrictions, but noting that $p_i p_j = p_j p_i$, we are imposing at most $10 + 9 + 8 = 27$ restrictions. Doing the computations in Maple we obtain that the space $\mathcal{E}$ of linear functionals satisfying all the above conditions is indeed a $35-27 = 8$-dimensional space.

We fix a monomial basis
$$\MM_{\bx} = \{x_0^2, x_0x_1, x_0x_2, x_0x_3, x_1^2, x_1x_2, x_1x_3, x_2^2, x_2x_3, x_3^2\}$$
for the 10-dimensional space $H_{4,2}$ and a basis $\{\ell_1, \dots, \ell_8\}$ for $\mathcal{E}$, and we write a general element of $\mathcal{E}$ as $\ell = t_1 \ell_1 + \dots + t_8 \ell_8$, $t_i \in \R$, $1 \le i \le 8$. Let $\Qmat(\bt) = \Qmat(t_1, \dots, t_8)$, be the matrix associated to the quadratic form $Q_\ell$ for the monomial basis $\MM_{\bx}$ (the entries of $\Qmat(\bt)$ are linear forms in $\{t_1, \dots, t_8\}$). We compute values of $t_1, \dots, t_8$ so that the resulting quadratic form is positive semidefinite and has kernel of low dimension. To achieve this, we compute a numerical approximation using SEDUMI \cite{SEDUMI} and deduce from it exact values. SEDUMI uses interior point methods, so in general the solution will be a solution of maximum rank. The solution found has rank 4 and kernel of dimension 6. Inspecting the numerical solution, we fix the following values:
$$
t_1 := 0, t_4 := 0, t_7 := 0, t_5 := 1, t_8 := 1, t_3 := 1/2.
$$

After fixing these values we can again check with SEDUMI that there exist values for the remaining two variables such that the resulting quadratic form if positive semidefinite.
Let $\tilde \Qmat$ be the new approximation obtained with SEDUMI. In this case, it is not easy to deduce exact values for the two remaining unknowns from the approximations given by SEDUMI. Computing the rank of the principal submatrices of $\tilde \Qmat$, we observe that $\tilde \Qmat_{[2,3] \times [2,3]}$ has rank 1 and $\tilde \Qmat_{[5,6,7] \times [5,6,7]}$ has rank 2.
Hence we force the determinants of the submatrices $\Qmat(\bt)_{[2,3] \times [2,3]}$ and $\Qmat(\bt)_{[5,6,7] \times [5,6,7]}$ to be 0.
Solving in Maple the equations we obtain
$$t_6 := \beta, t_2 := -\frac12 \alpha -\frac12 \beta,$$
where $\alpha$ is a root of $X^3 - 2$ as before and $\beta$ is a root of $X^2 + \alpha^2 X + (1 - \alpha^2)$ (and we can verify $\beta \not\in \Q(\alpha)$).

Applying these substitutions we obtain a matrix $\Qmat_{\bx} \in \Q(\alpha, \beta)^{10 \times 10}$ and we can verify by symbolic computations that this matrix is positive semidefinite. Its kernel has dimension 6, so every polynomial in a sum of squares decomposition of $f$ is a linear combination of 6 fixed polynomials.

We can group the generators of the kernel of $\Qmat_{\bx}$ in two sets supported in complementary sets of monomials. The first set is
$$
\arraycolsep=1.4pt
\begin{array}{r*{10}{l}}
u_1 &= (0, &0, &1, &-\alpha \beta, &0, &0, &0, &0, &0, &0)_{\MM_{\bx}}, \\
u_2 &= (0, &1, &0, &\alpha + \beta, &0, &0, &0, &0, &0, &0)_{\MM_{\bx}}
\end{array}
$$
(note that $p_3 = (0, 4 \alpha, 4, 4 \alpha^2, 0, 0, 0, 0, 0, 0, 0)_{\MM_{\bx}} = 4 o_1 + 4 \alpha o_2$).

The second set is
$$
\arraycolsep=1.4pt\def\arraystretch{2.5}
\begin{array}{r*{10}{l}}
u_3 &= (-\alpha + 2, &0,&0,&0, &\dfrac{\alpha-4}{4\alpha^2-2}, &0, &0, &0, &0, &1)_{\MM_{\bx}}, \\
u_4 &= (-2\alpha^2 + 1, &0,&0,&0, &\dfrac{1}{2}, &0, &0, &0, &1, &0)_{\MM_{\bx}}, \\
u_5 &= (\alpha^2 + 2\alpha + 2\beta + 2, &0,&0,&0, &\dfrac{\alpha^2-4\alpha}{4\alpha^2-2}, &0, &1, &0, &0, &0)_{\MM_{\bx}}, \\
u_6 &= (-\alpha + 2, &0,&0,&0, &\dfrac{-\alpha+4}{4\alpha^2-2}, &1, &0, &0, &0, &0)_{\MM_{\bx}}.
\end{array}
$$

We now prove that there exists no non-trivial combination of these vectors such that the resulting vector has rational coordinates.

\begin{lemma} With notations as above, let $W = \langle u_1, u_2, u_3, u_4, u_5, u_6 \rangle_{\R} \subset \R^{10}$. Then $W \cap \Q^{10} = \{0\}$.
\end{lemma}

\begin{proof} Since $\{u_1, u_2\}$ and $\{u_3, \dots, u_6\}$ have complementary non-zero coordinates, we analyze separately the spaces generated by each of these  sets.

We show first that there is no non-trivial linear combination of $u_1$ and $u_2$ with rational entries. If $c_1 u_1 + c_2 u_2$ has rational coordinates, by looking at the second and third coordinates, we immediately see that $c_1, c_2 \in \Q$. Suppose now $c_1 (-\alpha \beta) + c_2 (\alpha + \beta) = q \in \Q$. Then
$$
\beta = \frac{q - c_2 \alpha}{c_2 - c_1 \alpha},
$$
but $\beta$ has degree 6 over $\Q$ and the right  hand size is in $\Q(\alpha)$, which is an extension of $\Q$ of degree 3 (note that the denominator does not vanish because $c_1,c_2$ cannot be both zero). Hence, $\{c_1 u_1 + c_2 u_2\} \cap \Q^{10} = \{0\}$.

We show now that there is no non-trivial linear combination
$$c_3 u_3 + c_4 u_4 + c_5 u_5 + c_6 u_6$$
such that the resulting vector has rational coordinates. Otherwise, looking at the 6th, 7th, 9th and 10th coordinates, we see that $c_i \in \Q$, $3 \le i \le 6$. If we look at the first coordinate, $\beta$ only appears in $u_5$ and $\beta \not\in \Q(\alpha)$, hence $c_5 = 0$. Similarly $c_4 = 0$ because after setting $c_5 = 0$, $\alpha^2$ only appears in $u_4$ and $\{\alpha, \alpha^2\}$ are linearly independent over $\Q$. We now look at $c_3 u_3 + c_6 u_6$. To cancel $\alpha$ in the first coordinate, it must be $c_6 = -c_3$, but this gives a non-rational 5th coordinate in $c_3 u_3 -c_3 u_6 = c_3 \frac{\alpha-4}{2\alpha^2-1}$. Hence, $\{c_3 u_3 + \dots + c_6 u_6\} \cap \Q^{10} = \{0\}$.
\end{proof}

This finishes the proof of Proposition \ref{prop:block1}.

\begin{remark}
We observe that the $8$th coordinate (corresponding to $x_2^2$) is zero in all the 6 vectors in the kernel. That is, $x_2^2$ cannot appear in any polynomial in an SOS decomposition of $f$, which is clear since $x_2^4$ is not a monomial in $f$.
\end{remark}

\subsection{A strictly positive polynomial in the boundary of the SOS cone}
\label{section:g}
The second polynomial that we will use as a building block is an example of a strictly positive polynomial in the boundary of $\Sigma_{4,4}$. The concrete example we use was obtained in \cite{capco}, following the construction in \cite{blekherman}.

We take
\begin{align*}
q_1 &= y_0^2-y_3^2, \\
q_2 &= y_1^2-y_3^2, \\
q_3 &= y_2^2-y_3^2, \\
q_4 &= -y_0^2 - y_0y_1 -y_0y_2 + y_0y_3 - y_1y_2 + y_1y_3 + y_2y_3,
\end{align*}
and $g = q_1^2 + q_2^2 + q_3^2 + q_4^2$.
Following \cite{capco}, we know that the SOS decomposition of $g$ is unique up to orthogonal transformations. In particular, there is a form $\ell \in \Sigma_{4,4}^*$ such that the kernel of the associated quadratic form $Q_{\ell}$ is $\langle q_1, q_2, q_3, q_4 \rangle$. For the monomial basis
$$\MM_{\by} = \{y_0^2, y_0y_1, y_0y_2, y_0y_3, y_1^2, y_1y_2, y_1y_3, y_2^2, y_2y_3, y_3^2\},$$ the matrix associated to $Q_{\ell}$ (see also \cite{laplagneUMA}) is
$$
\Qmat_{\by} = \begin{pmatrix}
6 & -1 & -1 & 1 & 6 & -1 & 1 & 6 & 1 & 6 \\
-1 & 6 & -1 & 1 & -1 & -1 & 1 & -1 & 1 & -1 \\
-1 & -1 & 6 & 1 & -1 & -1 & 1 & -1 & 1 & -1 \\
1 & 1 & 1 & 6 & 1 & 1 & -1 & 1 & -1 & 1 \\
6 & -1 & -1 & 1 & 6 & -1 & 1 & 6 & 1 & 6  \\
-1 & -1 & -1 & 1 & -1 & 6 & 1 & -1 & 1 & -1 \\
1 & 1 & 1 & -1 & 1 & 1 & 6 & 1 & -1 & 1 \\
6 & -1 & -1 & 1 & 6 & -1 & 1 & 6 & 1 & 6 \\
1 & 1 & 1 & -1 & 1 & 1 & -1 & 1 & 6 & 1 \\
6 & -1 & -1 & 1 & 6 & -1 & 1 & 6 & 1 & 6
\end{pmatrix}.
$$

\subsection{Combination of the two examples.}

We define
$$r = x_2^2 - y_2^2$$  and $h(\bx, \by) = f + g + r^2$.

\begin{theorem}
The polynomial
$$h(x_0, x_1, x_2, x_3, y_0, y_1, y_2, y_3) = f(x_0, x_1, x_2, x_3) + g(y_0, y_1, y_2, y_3) + r(x_2, y_2)^2$$
is a strictly positive polynomial with rational coefficients that can be decomposed as the sum of squares of polynomials in $\Q(\sqrt[3]{2})[\bx, \by]$ but cannot be decomposed as the sum of squares of polynomials in $\Q[\bx, \by]$.
\end{theorem}

\begin{proof}
We first verify that $h$ is strictly positive. If $g(\bx, \by) = 0$, by Section \ref{section:g}, it must be $y_i = 0$ for $0 \le i \le 3$. Now $r(x_2, y_2) = 0$ and hence $x_2 = 0$. But $f$ has no non  trivial root with $x_2 = 0$.

We now prove that $h$ does not allow a rational decomposition. Our strategy will be to construct a form $\ell \in \Sigma_{8,4}^*$ that vanishes in $h$ and such that the kernel of the associated quadratic form contains no non-trivial polynomial with rational coefficients in the $\bx$-variables.

We proceed in a similar way as in Section \ref{section:block1}. The linear form $\ell$ must satisfy $\ell(p w) = 0$ for all $p \in \{p_1, p_2, p_3, q_1, q_2, q_3, q_4, r\}$ and all $w \in H_{8,2}$. In this case, the space of forms satisfying these restrictions has dimension 70. We add as additional restrictions that the associated quadratic form $Q_{\ell}$ restricted to the $\bx$-variables is $Q_{\bx}$ and restricted to the $\by$-variables is $Q_{\by}$.
This reduces the dimension of the affine space of forms to 62.

Setting the monomial basis
$
\MM_{\bxy} = \{x_0^2, x_0x_1, x_0x_2, x_0x_3, x_1^2, x_1x_2, x_1x_3,$ $x_2^2$, $x_2x_3, x_3^2, y_0^2, y_0y_1, y_0y_2, y_0y_3, y_1^2, y_1y_2, y_1y_3, y_2^2$, $y_2y_3$, $y_3^2$, $x_0y_0$, $x_0y_1$, $x_0y_2$, $x_0y_3$, $x_1y_0$, $x_1y_1, x_1y_2, x_1y_3, x_2y_0, x_2y_1, x_2y_2, x_2y_3, x_3y_0, x_3y_1, x_3y_2, x_3y_3\}
$,
we get a matrix $\Qmat$ associated to the quadratic form whose entries are linear affine in 62 variables. Setting the 62 variables to 0, we get
an exact matrix $\Qmat_{\bxy}$ and we can verify in Maple that it is positive semidefinite. The kernel of $\Qmat_{\bxy}$ has dimension 14. Computing generators in Maple, we see that the kernel can be generated by the 6 polynomials $\{u_1, \dots, u_6\}$ generating the kernel of $\Qmat_{\bx}$, plus the polynomials  $q_1, q_2, q_3, q_4, r$ and 3 new polynomials:
\begin{align*}
s_1 &= x_0y_0 + x_0y_1, \\
s_2 &= x_1y_0 + x_1y_1, \\
s_3 &= x_3y_0 + x_3y_1.
\end{align*}

We conclude that every polynomial in an SOS decomposition of $h$ is a linear combination of these 14 polynomials.

Assume now that there is a rational SOS decomposition of $h$ and let $p$ be a polynomial in that decomposition, then $p$ is a linear combination of the 14 generators of the kernel of $\Qmat_{\bxy}$.
We have already seen that there is no non-trivial polynomial with rational coefficients in the kernel of $\Qmat_{\bx}$, and we verify that the monomials appearing in $q_1, q_2, q_3, q_4, r, s_1, s_2, s_3$ are all different from the monomials in the polynomials $u_1, \dots, u_6$, hence the coefficients of $u_1, \dots, u_6$ in $p$ must all be zero. This implies that $p$ has no term $x_0^2$. Since this is true for all polynomials in a rational decomposition of $h$, the monomial $x_0^4$ cannot appear in $h$, a contradiction since $x_0^4$ does occur in $h$.

\end{proof}

\begin{remark}
If we attempt to construct other examples, numerical experiments using SEDUMI suggest that if we replace $r = x_2^2 - y_2^2$ by other polynomials such as $x_2^2 - y_1^2$ or $x_2^2 - 2y_2^2$, the resulting polynomials are also examples of strictly positive polynomials that are $\R$-SOS but not $\Q$-SOS. However, if we use instead $r = x_2^2 + y_2^2$, the resulting polynomial is in the interior of the SOS cone, and hence it is a $\Q$-SOS.
\end{remark}

\section{A nonsingular projective variety}
\label{section:hyper}

In \cite[Section 5.1]{scheiderer}, C. Scheiderer asks if there exists a $\Q$-form which is an $\R$-SOS but not a $\Q$-SOS such that the hypersurface it defines is nonsingular. This implies that the polynomial has no real zeros, because any real zero of a non-negative polynomial is a singular point of the hypersurface, and that the polynomial is irreducible, because the intersection of the hypersurfaces defined by the factors is in the singular locus. Hence, as the author mentions, being nonsingular is a common sharpening of both properties. That is, the motivation for this question is to find an example with as little structure as possible, so that it cannot be detected by a priori methods (e.g. computing the singularities) if the polynomial can be decomposed over $\Q$.

Using Singular \cite{DGPS} (see \cite[Worksheet B]{nonRatWorksheets}), we can verify by a Groebner basis computation that the hypersurface $\{h = 0\}$, where $h$ is the polynomial we defined in the previous section, contains no singular points, that is:
$$
\sqrt{\left\langle f, \frac{\partial f}{\partial{x_0}}, \dots, \frac{\partial f}{\partial{y_3}}\right\rangle} = \langle x_0, x_1, x_2, x_3, y_0, y_1, y_2, y_3 \rangle,
$$
hence giving a positive answer to that question.

\bibliographystyle{plain}
\bibliography{../../rationalSOS}

\end{document}